\documentclass{amsart}

\usepackage{hyperref}
\usepackage{tikz-cd}
\usepackage{todonotes}

\hypersetup{urlcolor=blue, citecolor=blue, linkcolor=blue, colorlinks=true}

\usepackage{enumitem}
\usepackage{array}
\usepackage{amssymb}
\usepackage{mathtools}
\usepackage{faktor}
\usepackage{xcolor}
\usepackage{tikz}
\usetikzlibrary{arrows}
\usetikzlibrary{decorations.pathmorphing}
\tikzset{snake it/.style={decorate, decoration=snake}}
\usepackage{tikz-cd}
\usepackage{amsthm}
\usepackage{thmtools}
\usepackage{tcolorbox}
\tcbuselibrary{listings,breakable}
\usepackage[utf8]{inputenc}
\usepackage{csquotes}
\usepackage[english]{babel}

\usepackage[scr=rsfs]{mathalpha}
\usepackage[capitalise,nameinlink]{cleveref}

\newcolumntype{L}{>{\scriptstyle}l}
\newcolumntype{C}{>{\scriptstyle}c}
\newcolumntype{R}{>{\scriptstyle}r}

  \definecolor{grau1}{RGB}{100, 100, 100}
  \definecolor{grau2}{RGB}{150, 150, 150}
  \definecolor{grau3}{RGB}{200, 200, 200}

\DeclareMathAlphabet{\mathpzc}{OT1}{pzc}{m}{it}

\usepackage{newunicodechar}

\newtheorem{bigthm}{Theorem}

\newtheorem{thm}{Theorem}[section]
\newtheorem*{thm*}{Theorem A (ii)'}
\newtheorem{lemma}[thm]{Lemma}

\theoremstyle{definition}

\theoremstyle{remark}

\newtheorem{remark}[thm]{Remark}

\newcommand{\nocontentsline}[3]{}
\newcommand{\tocless}[2]{\bgroup\let\addcontentsline=\nocontentsline#1{#2}\egroup}

\setenumerate[1]{leftmargin=*,labelindent=0pt,label=(\roman*),}

\newcommand{\pref}[2]{\hyperref[#1]{#2 \ref*{#1}}}
\newcommand{\id}{\ensuremath{\operatorname{id}}}

\newcommand{\Dirac}{\mathcal D}

\newcommand{\scal}{\ensuremath{\operatorname{scal}}}
\newcommand{\Ric}{\ensuremath{\operatorname{Ric}}}
\newcommand{\tr}{\ensuremath{\operatorname{tr}}}

\newcommand{\ind}{\ensuremath{\operatorname{Ind}}}

\newcommand{\pr}{\mathrm{pr}}

\DeclarePairedDelimiter{\scpr}{\langle}{\rangle}

\newcommand{\eps}{\varepsilon}
\newcommand{\norm}[1]{\left\|#1\right\|}
\newcommand{\abs}[1]{\left|#1\right|}
\newcommand{\bigabs}[1]{\Bigl|#1\Bigr|}

\newcommand{\calR}{\mathcal{R}}
\newcommand{\calD}{\mathcal{D}}

\newcommand{\bbZ}{\mathbb{Z}}
\newcommand{\bbQ}{\mathbb{Q}}
\newcommand{\bbR}{\mathbb{R}}

\newcommand{\scrM}{\mathscr{M}}
\newcommand{\scrE}{\mathscr{E}}
\newcommand{\scrW}{\mathscr{W}}

\newcommand{\bt}{{\mathbf{t}}}
\newcommand{\bn}{\mathbf{n}}
\newcommand{\bOne}{\mathbf{1}}
\newcommand{\hatotimes}{\widehat{\otimes}}
\newcommand{\scrS}{\mathscr{S}}

\newcommand{\ahat}{\hat{\mathcal{A}}}

\newcommand{\torus}[1]{\mathbb{T}^{#1}}

\newcommand{\op}{{\mathrm{op}}}
\newcommand{\ch}{{\mathrm{ch}}}

\newcommand{\spinor}{\Sigma}

\usepackage{amsrefs}

\title[Scalar Curvature Rigidity for Products]{Scalar Curvature Rigidity for Products of convex Hypersurfaces and even dimensional Manifolds}

\date{\today}

\begin{document}

\author[Georg Frenck]{Georg Frenck}
\address[Georg Frenck]{Universität Augsburg, Universitätsstr.~14, 86159 Augsburg, Germany}
\email{\href{mailto:georg.frenck@math.uni-augsburg.de}{georg.frenck@math.uni-augsburg.de}}
\email{\href{mailto:math@frenck.net}{math@frenck.net}}
\urladdr{\href{http://frenck.net/Math}{Frenck.net/Math}}

\author[Thomas Schick]{Thomas Schick}
\address[Thomas Schick]{Mathematisches Institut, Universität Göttingen,
  Bunsenstr.~3, 37073 Göttingen, Germany}
\email{\href{mailto:thomas.schick@math.uni-goettingen.de}{thomas.schick@math.uni-goettingen.de}}
\urladdr{\href{https://topologie.math.uni-goettingen.de/tschick/index.html}{topologie.math.uni-goettingen.de/tschick/}}

\author[Lukas Schönlinner]{Lukas Schönlinner}
\address[Lukas Schönlinner]{Universität Augsburg, Universitätsstr.~14, 86159 Augsburg, Germany}
\email{\href{mailto:lukas.schoenlinner@uni-a.de}{lukas.schoenlinner@uni-a.de}}
\urladdr{\href{https://www.uni-augsburg.de/de/fakultaet/mntf/math/prof/diff/team/lukas-schonlinner/}{uni-augsburg.de/de/fakultaet/mntf/math/prof/diff/team/lukas-schonlinner/}}

\begin{abstract}
  We give a proof of scalar curvature rigidity in the spirit of Llarull and Goette-Semmelmann for products of strictly convex hypersurfaces in Euclidean space and nonnegatively curved spaces with non-vanishing Euler-characteristic.
  Our proof is based on the Fredholm family index theorem. This recovers corresponding results of Lockman-Zeidler where Clifford-linear (family) index theory is used.
\end{abstract}

\maketitle

\section{Introduction}

The classical scalar curvature rigidity theorem of Llarull \cite{Llarull} states that the round metric on an even-dimensional sphere is scalar curvature extremal and rigid. 
Goette and Semmelmann \cite{GS02} generalized this as follows: if $(M,g)$ is a closed smooth spin manifold  with positive curvature operator and $f\colon (W,g_W)\to (M,g_M)$ is a smooth area non-increasing maps satisfying $\scal_W\ge\scal_M\circ f$, then equality holds, provided the Euler characteristic of $M$ is non-zero.
If in addition $\scal_M > 2\Ric_M >0$ then $f$ is a Riemannian covering.

\medskip

Their proof uses Fredholm index theory for twisted Dirac operators, and this is where the condition on the Euler characteristic of $M$ crucially enters, as it corresponds to the K-theory cycle given by the spinor bundle.
Using a suspension trick, one obtains a proof of the corresponding result for odd-dimensional round spheres (compare \cite{Llarull} together with \cite{BaerBrendleHankeWang} for a complete argument).
To obtain a more conceptual proof in this case, while at the same time covering more general situations, one can work with an odd-degree K-theory class to obtain a non-trivial index.
This was achieved by Li--Su--Wang \cite{LiSuWang} and Bär--Ziemke \cite{BaerZiemke} using the spectral flow for strictly convex hypersurfaces in $\bbR^{n+1}$.

\medskip

Recently, Lockman--Zeidler presented another proof via Clifford-linear and family index arguments in \cite{LockmanZeidler}.
A key advantage of their approach is that it provides a uniform proof of both the even- and odd-dimensional cases of Llarull's Theorem.
Their main result is a scalar curvature rigidity result in which the target $M$
is the Riemannian product of strictly convex hypersurfaces of positive
curvature and a flat torus.

\medskip

In this article, we prove a similar scalar curvature rigidity result, the precise formulation is given in \cref{thm:main} below.
This result was obtained independently of \cite{LockmanZeidler}.
The main differences between our approach and theirs are the following:
\medskip
\begin{enumerate}[itemsep=4pt]
  \item We work with the classical Fredholm index (and Fredholm family index), as
  discussed, for example, in the textbook \cite{LawsonMichelsohn}, possibly upon taking the product with the round $S^3$. 
  This should allow a rather direct extension to
  settings with lower regularity of the comparison metric and the comparison map
  due to the existing Fredholm index theory for low regularity Dirac operators.
  \item While the construction of the family of operators in \cite{LockmanZeidler} has an algebraic flavour, our approach is more geometric.
  Our family of operators is obtained by pulling back from the higher-dimensional round sphere via an explicit family of smooth maps $S^n\to S^{n+1}$.
  \item Our target manifold is allowed to be the Riemannian product of any $(N,g)$
  with non-negative curvature operator and non-zero Euler characteristic with finitely many convex closed hypersurfaces of Euclidean space. 
  This yields a full extension of the rigidity result of Goette-Semmelmann.
\item \cite{LockmanZeidler}*{Theorem A} allows for products with tori as well, which requires additional arguments
  and case distinctions, both in the statement and in the proof. To keep our exposition short, we decided not to include this part.
\end{enumerate}

\subsection*{Acknowledgments}
We would like to thank Thomas Tony for fruitful discussions. This work was partially supported by the DFG-SPP 2026 \enquote{Geometry at infinity}.

\section{Scalar curvature rigidity for products}


\begin{bigthm}\label{thm:main}
  For $k\ge1$ let $n_1,\dots,n_k\ge3$ be odd natural numbers and let $S_i\subset\bbR^{n_i+1}$ be closed, convex hypersurfaces with induced metric $g_i$.
  Furthermore, let $(N,g_N)$ be a Riemannian spin manifold with nonvanishing Euler-characteristic and nonnegative curvature operator.
  Then for any smooth map 
  \begin{equation*}
    f\colon (W,g)\to (M, g_M)\coloneqq (N\times S_1\times\dots\times S_k,\  g_N\oplus g_1\oplus \dots\oplus g_k)
  \end{equation*}
  from a closed connected Riemannian spin manifold $(W,g)$ such that $\scal_{g}\ge\scal_{g_M}\circ f$ and $\deg(f)\neq 0$, and for which the respective compositions $\pr_N\circ f$ and $\pr_{S_i}\circ f$ are area non-increasing, we have:
  \begin{enumerate}
    \item $\scal_g = \scal_{g_M}\circ f$.
    \item if $\scal > 2\Ric >0$ holds for $N$ and all $S_i$, then $f$ is a Riemannian covering. 
  \end{enumerate}
\end{bigthm}

\begin{remark}
  \begin{enumerate}
    \item The condition that the composition with the respective projections is area non-increasing is weaker than the entire map being area non-increasing as the product of two area non-increasing maps need not be area non-increasing.
    \item For any convex hypersurface $S\subset \bbR^{n+1}$, $n\ge3$, the estimate $\scal > 2\Ric >0$ is satisfied if and only if $S$ is strictly convex.
  \end{enumerate}
\end{remark}

\begin{remark}
  The result leaves a number of obvious open questions, two of which we want
  to stress here:
  \begin{enumerate}
    \item Most obviously, it is remarkable that for even dimensional sphere
    factors, one can deal with any Riemannian metric with nonnegative curvature
    operator, whereas for odd dimensional manifolds, at the moment we have to
    assume that the metric is induced from a strictly convex embedding into
    Euclidean space of codimension $1$. We still lack a method to
    represent nontrivial classes in odd K-theory by cycles which are coupled sufficiently to the
    Riemannian metric in the general case.
    \item Due to the fact that the condition on the map can be phrased as a
    purely metric condition, Gromov asked whether one can lower the regularity of
    the comparison map and still get the rigidity result. This was
    answered positively in \cites{CecchiniHankeSchick,
    CecchiniHankeSchickSchoenlinner} (a different approach is given in
    \cite{LeeTam}). It remains an interesting task to
    extend our method to such low regularity settings, as well. 
  \end{enumerate}
\end{remark}


\section{Proof of \texorpdfstring{\cref{thm:main}}{Theorem \ref*{thm:main}}}

\noindent Like Llarull's original proof for rigidity of round spheres, the proof of \cref{thm:main} consists of two parts: an index-theoretic one and a geometric one. 
We begin with the index-theoretic part, which is based on the family index theorem for twisted Dirac operators.

\medskip

We first recall the family index theorem for families of twisted Dirac operators.
Let $B$ be a Hausdorff space and let $\pi\colon\scrM\to B$ be a smooth fibre bundle of closed even-dimensional Spin manifolds with $T\scrM$ the associated vertical tangent bundle equipped with a Riemannian metric and a Spin structure.
Furthermore, let $\scrE\to B$ be a fiber bundle whose fiber over $b\in B$ is a vector bundle $E_b\to M_b \coloneqq \pi^{-1}(b)$ equipped with a smooth family of metric connections $\nabla^E$.
Consider the family of twisted Dirac operators given by
\[
\Dirac^+_{\scrE,b}\coloneqq\Dirac^{+}_{M_b}\otimes\id + \id\otimes\nabla^{E_b}\colon C^\infty(\spinor_+M_b\otimes E_b)\to C^\infty(\spinor_-M_b\otimes E_b).
\]
The K-theoretical family index theorem states the following.
\begin{thm}[\cite{LawsonMichelsohn}*{Corollary III.15.5}]
  $\ind(\Dirac^+_\scrE) = \pi_!(\ch(\scrE)\widehat{A}(T\scrM))\in K(B)$, where $\pi_!\colon K(\scrM)\to K(B)$ is the pushforward in $K$-theory.
\end{thm}
To obtain a numerical invariant, we can integrate the Chern character of the index class over $B$ and obtain
\[
\int_B\ch(\ind(\Dirac^+_\scrE)) = \int_B \pi_!\left(\ch(\scrE)\widehat{A}(T\scrM)\right) \in \bbQ.
\]
Let us turn to the proof of \cref{thm:main}.
Without loss of generality, we may assume that $\dim(W)$ is even, 
Otherwise, we can take the product with a round $3$-sphere and consider the map $f\times\id_{S^3}\colon W\times S^3\to M\times S^3$.

\medskip

For  a strictly convex hypersurface  $S_i\subset\bbR^{n_i+1}$ we denote by
$\nu_i\colon S_i\to S^{n_i}$ the Gauss map assigning to $x$ the normal vector of $S_i$ at $x$.
Furthermore, we write $\gamma_i$ for the round metric on $S^{n_i+1}$ and $\gamma\coloneqq\oplus\gamma_i$.
For the multi-index $\bn = (n_1, \ldots , n_k)$ we denote
\begin{equation*}
  \scrS^\bn \coloneqq S^{n_1}\times\dots\times S^{n_k}.
\end{equation*}
Set 
\begin{equation*}
  \nu\coloneqq \nu_1\times \ldots \nu_k \colon S_1 \times \ldots \times S_k \to \scrS^\bn.
\end{equation*}
Fix a $2\pi$-periodic function $\rho\colon\bbR\to [-1,1]$ satisfying the following:
\begin{enumerate}
  \item $\rho(0)=1$ is the only maximum of $\rho$ inside $[-\pi,\pi]$,
  \item $\rho(x+\pi)=-\rho(x)$,
  \item $\rho \equiv 0$ on $[\tfrac{\pi}{2}-\eps,\tfrac{\pi}{2}+\eps]$ for some $\eps>0$
  \item the following, $2\pi$-periodic function is smooth: \[\sigma(t)\coloneqq\begin{cases}
    \sqrt{1-\rho(t)^2} &\text{ for }t\in[0,\pi] + 2\pi\bbZ\\
    -\sqrt{1-\rho(t)^2} &\text{ for }t\in[\pi,2\pi] + 2\pi\bbZ
  \end{cases}\]
\end{enumerate}
Such a function can, for example, be constructed as the product $\cos\cdot\lambda$, where $\lambda$ is a smooth bump function supported near $\pi\cdot\bbZ$, see \cref{fig:graph-of-rho}.
\begin{figure}
  \begin{tikzpicture}
    \node at(3.3,0) {\includegraphics[width=.9\textwidth, trim=200 300 200 300, clip]{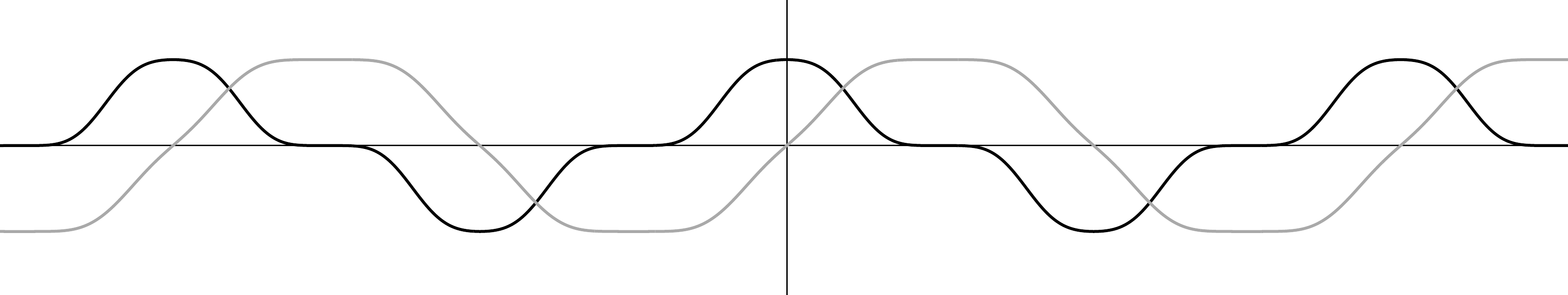}};
    \node at(0,-.7) {$\rho$};
    \node at(5.5,.8) {\color{gray}{$\sigma$}};
  \end{tikzpicture}
  \caption{The functions $\rho$ and $\sigma$.}\label{fig:graph-of-rho}
\end{figure}
Using the identification $\bbR/2\pi\bbZ\cong S^1$ and letting  $p_i\in S^{n_i}$ be the north pole, we define the map
\[\phi_i\colon S^{n_i}\times S^1\to S^{n_i+1},\quad (x,\ t)\mapsto \begin{cases}
    \left(\rho(t)\cdot x,\ \sigma(t)\right) &\text{ for }t\in[-\pi/2,\pi/2] + 2\pi\bbZ\\
    \left(\rho(t)\cdot p_i,\ \sigma(t)\right)&\text{ for }t\in[\pi/2,3\pi/2] + 2\pi\bbZ
  \end{cases}
\]
For \(t\in[-\pi/2,\pi/2]\), the map \(\phi_i\) describes a family of latitudinal embeddings of \(S^{n_i}\), moving from the south pole to the north pole; for \(t\in[\pi/2,3\pi/2]\), it follows a longitudinal line from the north pole back to the south pole
Note that $\deg(\phi_i)=1$, since $(x,0)\not=(p_i,0)$ in $S^{n_i+1}$ has precisely one preimage.
For $t\in S^1$ we set $\phi_{i,t}\coloneqq \phi_i(\_, t)$ and we observe that $\phi_{i,2\ell\pi}$ is an embedding of the equator for all $\ell\in\bbZ$.
For $\bn + \bOne \coloneqq (n_1 + 1, \ldots, n_k + 1)$ we define a map
\begin{align*}
  \Phi\colon{}& \scrS^\bn\times \torus{k}\to \scrS^{\bn + \bOne}\\
    &(x_1,\dots,x_k,t_1,\dots,t_k)\mapsto (\phi_1(x_1,t_1),\dots,\phi_k(x_k,t_k))\\
  \Phi_M\coloneqq{}& \id_N\times(\Phi\circ(\nu\times\id_{\torus{k}}))\colon M\times\torus{k}\to N\times \scrS^{\bn + \bOne}.
\end{align*}
Since $\deg(\phi_{i})=1$, we have $\deg(\Phi_M) = 1$, too.
For $f$ as in the theorem we define 
\[
F \coloneqq f\times\id_{\torus{k}}\colon W\times\torus{k}\to M\times\torus{k}. 
\]
We note that the spinor bundle of a product of even-dimensional manifolds
$X,Y$ can be identified via the graded exterior tensor product $\Sigma(X\times Y) = \Sigma X\hatotimes\Sigma Y$.
Let
\begin{equation*}
  \left(\spinor \left( N\times \scrS^{\bn +\bOne}\right), \nabla^{\spinor (N\times \scrS^{\bn +\bOne})}\right)
\end{equation*}
be the spinor bundle over $N\times \scrS^{\bn +\bOne}$ with spinor connection associated with the metric $g_N\oplus \gamma$. 
We define the bundle 
\[
(E,\nabla^E)\coloneqq F^*\Phi_M^*\left(\Sigma( N\times \scrS^{\bn +\bOne}),\nabla^{\spinor (N\times \scrS^{\bn +\bOne})}\right)
\]
over $W\times\torus{k}$.
For a fixed tuple ${\bf t}\coloneqq(t_1,\dots,t_k)\in\torus{k}$ and $X\in TW$ we obtain
\begin{align}
  \begin{split}
    \nabla_X^{E,\bt}\coloneqq{}& (\nabla^E|_{W\times\{\bt\}})_X\\
      ={}& \Bigl(f^*(\id_N\times((\phi_{1, t_1}\times\dots\times\phi_{k, t_k})\circ\nu))^*\nabla^{\spinor (N\times \scrS^{\bn +\bOne})}\Bigr)_X\\
      ={}& f^*\Bigl(\nabla^{g_N}\hatotimes\id_{\Sigma\scrS^{\bn +\bOne}} + \id_{\spinor N}\hatotimes\nu^*(\phi_{1,t_1}\times\dots\times\phi_{k,t_k})^*\nabla^{\spinor \scrS^{\bn +\bOne}}\Bigr)_X.
  \end{split}
\end{align}
Let $E^\bt \coloneqq E\vert_{W\times \{\bt\}}$.
Since $n_i$ is odd, we have
\begin{align*}
  \Sigma \scrS^{\bn +\bOne}\cong{}& \Sigma S^{n_1+1}\hatotimes\dots\hatotimes\Sigma S^{n_k+1}\\
  \nabla^{\spinor \scrS^{\bn +\bOne}} ={}&\sum_{i} \id_{\Sigma S^{n_1+1}}\hatotimes\dots\hatotimes\nabla^{\spinor S^{n_i+1}}\hatotimes\dots\hatotimes \id_{\Sigma S^{n_k+1}}.
\end{align*}
For every $\bt\in \torus{k}$, let 
\begin{equation*}
  \Dirac_{E,\bt} \colon C^\infty(\spinor W\otimes E^\bt) \to C^\infty(\spinor W\otimes E^\bt)
\end{equation*}
be the Dirac operator associated with the metric $g$ and twisted with $(E, \nabla^{E,\bt})$.
Since $\dim W = \dim N + \sum_i n_i=\dim M$ is even, we have orthogonal decompositions
\begin{align*}
  \spinor W ={}& \spinor_+ W \oplus \spinor_- W\\
  E^\bt ={}& E^\bt_+ \oplus E^\bt_-
\end{align*}
into the $\pm 1$-eigenbundles of the respective volume element. 
Clifford multiplication on the twist bundle induces a $\bbZ/2\bbZ$-grading
\begin{align*}
  \left(\spinor W \otimes E^\bt\right)_+ ={}& \left(\spinor_+ W \otimes E^\bt_+\right) \oplus \left(\spinor_- W \otimes E^\bt_-\right) \\
  \left(\spinor W \otimes E^\bt\right)_- ={}& \left(\spinor_- W \otimes E^\bt_+\right) \oplus \left(\spinor_+ W \otimes E^\bt_-\right)
\end{align*}
and the Dirac operator satisfies
\begin{equation*}
  \Dirac_{E,\bt}^\pm \colon C^\infty((\spinor W\otimes E^\bt)_\pm) \to C^\infty((\spinor W\otimes E^\bt)_\mp).
\end{equation*}
We define 
\begin{align*}
  \scrW\coloneqq{}& W\times \torus{k}\to \torus{k}\\
  \scrE\coloneqq{}& E\to W\times \torus{k}\to \torus{k}\\
  \spinor \scrW \coloneqq{}& \pr_W^* \spinor W \to \scrW,
\end{align*}
which comes with a $\bbZ/2\bbZ$-grading
\begin{equation*}
  (\Sigma\scrW\otimes \scrE)_\pm = (\pr_W^*\spinor W \otimes E)_\pm.
\end{equation*}
Finally, we obtain a continuous family of Dirac operators
\begin{equation*}
  \Dirac_\scrE^\pm \colon C^\infty((\Sigma\scrW\otimes \scrE)_\pm) \to C^\infty((\Sigma\scrW\otimes \scrE)_\mp). 
\end{equation*}
We deduce from the family index theorem and the fact that the Dirac operator is diagonal with respect to the splitting inside the bundles $(\Sigma\scrW\otimes \scrE)_\pm$ that
\begin{align*}
   \int_{\torus{k}}\ch(\ind(\Dirac^+_\scrE)) ={}& \int_{\torus{k}}\ch(\ind\Dirac_{\scrE}^+\colon C^\infty(\Sigma\scrW_+\otimes \scrE_+) \to C^\infty(\Sigma\scrW_-\otimes \scrE_+)) \\
   &- \int_{\torus{k}}\ch(\ind\Dirac_{\scrE}^+\colon C^\infty(\Sigma\scrW_-\otimes \scrE_-) \to C^\infty(\Sigma\scrW_+\otimes \scrE_-)) \\
   ={}& \int_{\torus{k}}\pi_!(\ch(\scrE_+)\ahat(T\scrW)) - \int_{\torus{k}}\pi_!(\ch(\scrE_-)\ahat(T\scrW))\\
  ={}& \int_{\torus{k}\times W} (\ch(E_+)- \ch(E_-))\ahat(T\scrW)\\
  ={}& \int_{\torus{k}\times W}F^*\Phi_M^*\Bigl(\ch\left(\Sigma_+( N\times \scrS^{\bn +\bOne})\right)\\
    &\qquad\qquad\qquad\qquad - \ch\left(\Sigma_-( N\times \scrS^{\bn +\bOne})\right)\Bigr)\ahat(T\scrW)\\
    ={}& \deg(\Phi_M\circ F)\cdot \chi(N\times S^{n_1+1}\times\dots\times S^{n_k+1})\\
    ={}& \deg(F) \cdot 2^{k}\cdot\chi(N)\not=0,
\end{align*}
where the penultimate equality is the standard result that the graded spinor bundle
on a closed spin manifold represents the fundamental K-theory generator
(Poincar\'e dual to the fundamental class), multiplied with the Euler
characteristic. Its pullback by a map of degree $d$ hence represents the
fundamental K-theory class multiplied with $d$ times the Euler characteristic
of the cycle. Its pairing with the fundamental K-homology class is given by
the integral. For more details, compare \cite{GS02} or
\cite{LawsonMichelsohn}*{Section 6}.
Therefore, there exists a $\bt\in\torus{k}$ such that $\Dirac_{E,\bt}$ has non-trivial kernel.
This concludes the index-theoretic part of the argument.

\medskip

We now turn to the geometric part, beginning with the following lemma.
\begin{lemma}\label{LinAlgEstimate}
  Let $V,W,Z$ be real Euclidean vector spaces with $\dim V = \dim Z$. 
  Moreover, let $A\colon V\to W$ and $B\colon W\to Z$ be linear maps with $\norm{A}_{\rm op}=1$ and let $T\colon V\times Z\to \bbR$ be a bilinear map. 
  By the singular value decomposition, there are orthonormal bases $(w_\beta)$ and $(z_\beta)$ of $W$ and $Z$ and real numbers $b_\beta$ such that $Bw_\beta = b_\beta z_\beta$. 
  If we can arrange that $b_\beta \geq 0$ for all $\beta$, then
  \begin{equation*}
    \abs{\tr(T(-, BA-))} \leq \norm{T}_{\rm op} \sum_\beta b_\beta.
  \end{equation*}

  Moreover, equality can only occur if $b_\beta \abs{A^* w_\beta}_V = b_\beta$ for all $\beta$, where $A^*$ denotes the adjoint of $A$.
\end{lemma}

\begin{proof}
  Let $(v_\alpha)$ be any orthonormal basis of $V$. 
  Then we compute
  \begin{align*}
    \abs{\tr(T(-, BA-))} ={}&\bigabs{\sum_{\alpha} T(v_\alpha, BA v_\alpha)} \\
      ={}& \bigabs{\sum_{\alpha, \gamma} \scpr{BA v_\alpha, z_\gamma}_Z T(v_\alpha, z_\gamma)} \\
    ={}& \bigabs{\sum_{\alpha, \beta, \gamma} \scpr{Av_\alpha, w_\beta}_W \scpr{Bw_\beta, z_\gamma}_Z T(v_\alpha, z_\gamma)}  \\
    ={}& \bigabs{\sum_{\beta} b_\beta T\Bigl(\sum_\alpha \scpr{Av_\alpha, w_\beta}_Wv_\alpha, z_\beta\Bigr)} \\
    \leq& \norm{T}_{\rm op}\sum_{\beta} b_\beta \bigabs{\sum_\alpha {\scpr{Av_\alpha, w_\beta}_W}v_\alpha}_V\\
    ={}& \norm{T}_{\rm op}\sum_{\beta} b_\beta \abs{A^*w_\beta}_V \le \norm{T}_{\rm op}\sum_{\beta} b_\beta,
  \end{align*}
  where the last inequality uses $\|A^*\|_\op = \|A\|_\op\le 1$.
\end{proof}

\noindent Note that the Dirac operators $\Dirac_{E,\bt}$, $\bt\in\torus{k}$ satisfy the integrated Schrödinger--Lichnerowicz formula
\begin{align*}
  \bigl\|\calD_{E,\bt} u\bigr\|^2 ={}& \bigl\|\nabla^{\spinor W\otimes E^\bt} u\bigr\|^2 + \frac14\left(\scal_W u ,u\right) + \bigl(\calR^{E^\bt} u, u\bigr),
\end{align*}
where $\calR^{E^{\bt}}$ is the curvature operator of the bundle $\spinor W\otimes E^{\bt}$.
This curvature operator acts on $u =  \sigma\otimes \eta\in C^\infty(\spinor W\otimes E^{\bt})$ by
\begin{align*}
  \calR^{E^\bt} u ={}& \frac12\sum_{a\neq b} c(e_a)c(e_b)\sigma \otimes R^{E^\bt}_{e_a, e_b}\eta\\
\end{align*}
where $(e_1, \ldots, e_n)$ is an orthonormal frame of $TW$ and $R^{E^\bt}$ is the curvature tensor of $(E^\bt, \nabla^{E^\bt})$.
Note that $E^\bt$ is the pullback of the spinor bundle of the manifold
$P\coloneqq N\times \scrS^{\bn+\bOne}$ along the map $\alpha^\bt\coloneqq
\Phi_M(\_,\bt)\circ f$, hence we have 
\begin{align*}
  R^{E^\bt}_{e_a, e_b} ={}& \sum_{i,j}g_P(R^{(\alpha^\bt)^*\Sigma P}_{e_a, e_b}\eps_i,\eps_j) c^{P}(\eps_i)c^{P}(\eps_j)\\
    ={}&\sum_{i,j}g_P(R^{P}_{\alpha^\bt_* e_a, \alpha^\bt_* e_b}\eps_i,\eps_j) c^{P}(\eps_i)c^{P}(\eps_j),
\end{align*}
for $\eps_i$ an orthonormal basis of $TP$.
We have the following pointwise estimate for the curvature endomorphism $\calR^{E^\bt}$.

\begin{lemma}\label{CurvOpPointwiseEstimate}
  For every $\bt\in \torus{k}$ we have
  \begin{equation}\label{EqCurvOpPointwiseEstimate}
    \scpr{\calR^{E^\bt}u, u} \geq -\tfrac14 \scal_M\circ f \cdot \abs u^2
  \end{equation}
  for all $u\in \spinor W\otimes E^{\bt}$.
  Moreover, for all $p\in W$ such that $u(p)\neq 0$, we have that 
  \begin{enumerate}[label = \textup{(\roman*)}]
    \item the inequality is strict if $t_i$ is not a multiple of $\pi$ for at least one
      $i\in \{1,\dots,k\}$,
    \item equality occurs at $\bt\in(\pi\cdot\bbZ)^k$ if and only if $D_pf$ is an isometry.
  \end{enumerate}
\end{lemma}

\begin{proof}
    The curvature tensor of $M$ decomposes orthogonally as
  \begin{align*}
    R^{\Phi_M(\_,\bt)^*(N\times\scrS^{\bn + \bOne})} ={}& \Phi_M(\_,\bt)^*(R^{N}\oplus R^{\scrS^{\bn+\bOne}})\\
      ={}& \pr_N^*R^{N} \oplus \bigoplus_{i=1}^k \pr_i^*\nu_i^*\phi_{i,t_i}^*R^{S^{n_i+1}}.
  \end{align*}
  Therefore, if we write $\pr_N\colon M \to N$ and $\pr_i\colon M \to S_i$ for the respective projections and set $f_N \coloneqq \pr_N\circ f$, $f_i \coloneqq \pr_i\circ f$, $V_N = f_N^*\spinor N$ and $V_{i}  = ( \phi_{i,t_i}\circ \nu_i \circ f_i)^*\spinor S^{n_i+1}$, we get
  \begin{align*}
    R^{E^\bt}_{e_a, e_b}
      ={}& \sum_{k,\ell}g_N(R^{V_N}_{e_a,  e_b}\eps_k,\eps_\ell) c^N(\eps_k)c^N(\eps_\ell)\\
      &\qquad+ \sum_{i}\sum_{k,\ell }g_i(R^{V_i}_{e_a, e_b}\eps_k,\eps_\ell) c^M(\eps_k)c^M(\eps_\ell)
  \end{align*}
  Thus, we get a decomposition
  \begin{equation*}
    \calR^{E^\bt} = \calR^{V_N} + \calR^{V_1} + \ldots + \calR^{V_k}.
  \end{equation*}
  For $i \in \{1, \dots, k\}$ we define a bilinear map
  \begin{equation*}
    T_i\colon \Lambda^2T_pW\times \Lambda^2T_{(\phi_{i, t_i}\circ \nu_i \circ f_i)(p)}S^{n_i+1} \to \bbR,\quad (X,Y)\mapsto \scpr{(c(X)\otimes c(Y))u,u}.
  \end{equation*}
  Note that $\norm{T_i}_{\rm op} = \abs{u}^2$.
  Let 
  \begin{align*}
    A_i \coloneqq \Lambda^2 (f_i)_* &\colon \Lambda^2 T_pW \to \Lambda^2 T_{f_i(p)}S_i\\
    B_i \coloneqq \Lambda^2 (\phi_{i,t_i}\circ \nu_i)_* &\colon \Lambda^2 T_{f_i(p)}S_i \to \Lambda^2 T_{\phi_{i,t_i}(\nu_i(f_i(p)))}S^{n_i+1}. 
  \end{align*}
  and let $b^i_\beta$ be the singular values of $\Lambda^2(\nu_i)_*$. Then
  $B_i$ has singular values $\rho^2(t_i)b^i_\beta$ for $t_i\in[0,\pi]$ and $0$ otherwise, where $\rho$ is the fixed map from above, see \cref{fig:graph-of-rho}. By Gauss ``theorema egregium'' applied to the convex hypersurface $S_i$ we have
  \begin{equation}
    \label{eq:scal_calculation_for_hypersurface}
    \sum_\beta b^i_\beta = \scal_{S_i}\circ f_i.
  \end{equation}
  Since the curvature operator associated with $R^{S^{n_i+1}}$ is the identity, we get
  \begin{equation*}
    \scpr{\calR^{V_i}u, u} = \tfrac14\tr_g(T_i(-, A_iB_i -)).
  \end{equation*}
  From \cref{LinAlgEstimate} we deduce that
  \begin{equation*}
    \scpr{\calR^{V_i}u, u} \geq - \tfrac14\rho^2(t_i)\scal_{S_i}\circ f_i\abs{u}^2.
  \end{equation*}
  Turning to the $N$-factor, we consider
  \begin{equation*}
    T_N\colon \Lambda^2T_pW \times \Lambda^2T_{f_N(p)}N \to \bbR,\quad (X,Y)\mapsto \scpr{(c(X)\otimes c(Y))u,u},
  \end{equation*}
  which again has $\|T_N\|_{\op}=\|u\|^2$, and as before we define
  \begin{align*}
    A_N \coloneqq{} \Lambda^2 (f_N)_* &\colon \Lambda^2 T_pW \to \Lambda^2 T_{f_N(p)}N\\
    B_N \coloneqq{} R^N &\colon \Lambda^2 T_{f_N(p)}N \to \Lambda^2 T_{f_N(p)}N.
  \end{align*}
  Then we obtain
  \begin{equation*}
    \scpr{\calR^Nu, u} \geq -  \bigl|\tr(R^{f_N^*\Sigma N})\bigr|\abs{u}^2 \geq - \tfrac14\scal_N\circ f_N \abs{u}^2.
  \end{equation*}
  Putting everything together, we obtain
  \begin{align}\label{eq:final-estimate-curvature-endomorphism}
    \begin{split}
      \scpr{\calR^{E^\bt}u, u} \geq{}& -\tfrac14\Bigl(\scal_N\circ f_N + \sum_{i=1}^k \rho^2(t_i)\scal_{S_i}\circ f_i\Bigr)\abs{u}^2\\
      \geq{}& -\tfrac14\scal_M\circ f \abs{u}^2
    \end{split}
  \end{align}
  and the second inequality is strict if $u(p)\not=0$ and $t_i$ is not a multiple of $\pi$ for some $i\in\{1,\dots,k\}$.

  \medskip

  Assume now that $\bt \in (\pi\cdot\bbZ)^k$ and that we are in the equality case. 
  The following argument is a slight modification of \cite{GS02}*{Section 1.2}.
  Fix $i\in \{1, \ldots , k\}$.
  Let $(e_\ell)$ be a $g_i$-orthonormal basis of $TS_i$ and $(\eps_\ell)$ be an orthonormal basis for $TS^{n_i}$ such that $B_i(e_\ell\wedge e_m) = b_{\ell m} \eps_\ell \wedge \eps_m$, where $b_{\ell m} = g(R^{S_i} (e_\ell , e_m)e_m, e_\ell)$ are the singular values of $B_i$. 
  Let $\psi_i \coloneqq (f_i)_*\colon TW\to TS_i$ and let $\psi_i^*\colon TS_i\to TW$ be its adjoint.  
  Note that $\psi_i^*$ has the same singular values as $\psi_i$.
  From the equality case in \cref{LinAlgEstimate} we deduce that 
  \begin{align*}
    b_{\ell m} ={}& b_{\ell m}\abs{\left(\Lambda^2 \psi_i\right)^*e_\ell\wedge e_m }_g.
  \end{align*}
  Therefore, whenever $b_{\ell m}>0$, we have
  \begin{equation}\label{SingularValueEstimate}
    1 = \abs{\left(\Lambda^2 \psi_i\right)^*e_\ell\wedge e_m }_g = \left(\abs{\psi_i^*e_\ell}_g^2\abs{\psi_i^*e_m}_g^2 - g(\psi_i^*e_\ell, \psi_i^* e_m)^2\right)^{\frac12}.
  \end{equation}
  Let $a_1 \geq \ldots \geq a_{n_i}$ be the singular values of $\psi_i^*$ and let $v\in TS_i$  be such that $\abs v_{g_i}=1$ and $\abs{\psi_i^*v}_g= a_1$.
  Suppose, for a contradiction that $a_1 >1$.
  Since $a_1 a_i \leq 1$ for all $i\geq 2$, we get $a_i \leq \tfrac1{a_1} < 1$.
  Hence, the equality 
  \begin{equation*}
    1 = \abs{\left(\Lambda^2 \psi_i\right)^*e_\ell\wedge e_m }_g
  \end{equation*}
  can only occur if $v$ is not orthogonal to $\mathrm{span}(e_\ell, e_m)$.
  For any $\ell \in \{1, \ldots , n_i\}$ we have   
  \begin{equation}\label{eq:scalar-ricci-comparison}
    \sum_{s,t}b_{st} = \scal_{S_i}   >  \underbrace{2\Ric_{S_i}(e_\ell)}_{>0} = 2\sum_{m}b_{\ell m}.
  \end{equation}
  If we fix an $\ell$, this implies that there are indices $m, s, t\in\{1,\dots,n_i\}\setminus\{\ell\}$, such that $b_{\ell m}>0$ and $b_{st}>0$.
  If $s\not=m\not=t$, then
  \begin{equation*}
    v\not\perp {\rm span}(e_\ell, e_m)\cap {\rm span}(e_s, e_t) = \{0\},
  \end{equation*}
  which is a contradiction.
  If $s = m$ or $t = m$ we deduce from $\scal_{S_i}>2\Ric_{S_i}(e_m)$ and \eqref{eq:scalar-ricci-comparison} that there is another pair $(r,q)$ of indices which are distinct from $m$ such that $b_{rq}>0$.
  Then we get a contradiction from
  \begin{equation*}
    v\not\perp {\rm span}(e_\ell, e_m)\cap {\rm span}(e_s, e_t)\cap{\rm span}(e_r, e_q) = \{0\}.
  \end{equation*}
  Consequently, all singular values of $\psi_i^*$ are at most $1$. 
  Together with \eqref{SingularValueEstimate}, this implies 
  \begin{equation*}
    g(\psi_i^*e_\ell, \psi_i^* e_m) = \delta_{m\ell}.
  \end{equation*}  
  Repeating this argument for the $N$-factor and all $S_i$-factors, we conclude that $D_pf$ is an isometry.
\end{proof}

\noindent Inserting inequality \eqref{EqCurvOpPointwiseEstimate} into the integrated Schrödinger-Lichnerowicz formula, we obtain
\begin{align*}
  \bigl\|\Dirac_{E,\bt} u\bigr\|^2 \geq \tfrac14 \left(\left(\scal_W u,u\right) - \left((\scal_M\circ f) u, u\right)\right) \geq 0.
\end{align*}
Moreover, the first inequality is strict if $t_i$ is not a multiple of $\pi$ for some $i$.
Hence, we deduce that $\Dirac_{E, \bt}$ is invertible for all $\bt \notin (\pi\cdot\bbZ)^k$.
Therefore, if $0\not=u\in \ker\left(\Dirac_{E,\bt}\right)$, then $\bt\in(\pi\cdot\bbZ)^k$ and $u$ is a $\nabla^{\spinor W\otimes E_{\bt}}$-parallel spinor.
This implies that $\abs u$ is constantly nonzero and therefore we conclude from the equality case in \cref{CurvOpPointwiseEstimate} that $D_pf$ is an isometry for every $p\in W$, hence a local isometry.
Since $W$ is closed, this local isometry is a Riemannian covering onto its image and, as $\deg(f)\not=0$, $f$ is surjective. This finishes the proof of \cref{thm:main}.\hfill$\Box$


\end{document}